\documentclass[12pt,centertags,oneside]{amsart}
\usepackage{amsmath,amstext,amsthm,a4,amssymb,amscd}
\usepackage[mathscr]{eucal}
\usepackage{mathrsfs}
\usepackage{epsf}
\usepackage{typearea}
\usepackage{charter}

\usepackage[T1]{fontenc}

\usepackage[a4paper,width=16.2cm,top=3cm,bottom=3cm]{geometry}

\numberwithin{equation}{section}

\newcommand{\field}[1]{\mathbb{#1}}
\newcommand{\Z}{\field{Z}}
\newcommand{\R}{\field{R}}
\newcommand{\C}{\field{C}}
\newcommand{\N}{\field{N}}

 \def\cC{\mathscr{C}}
\def\cL{\mathscr{L}}

\def\cO{\mathscr{O}}

\def\mO{\mathcal{O}}

\def\bE{{\boldsymbol E}}

\newcommand{\be}{\begin{eqnarray}}
\newcommand{\ee}{\end{eqnarray}}
\newcommand{\ov}{\overline}
\newcommand{\wi}{\widetilde}
\newcommand{\var}{\varepsilon}
\newcommand{\bb}{\boldsymbol{b}}

\DeclareMathOperator{\End}{End}

\DeclareMathOperator{\Ker}{Ker}

\DeclareMathOperator{\Dom}{Dom}

\DeclareMathOperator{\rank}{rank}
\DeclareMathOperator{\Id}{Id}
\DeclareMathOperator{\supp}{supp}

\DeclareMathOperator{\inj}{inj}
\DeclareMathOperator{\spec}{Spec}
\newcommand{\spin}{$\text{spin}^c$ }
\newcommand{\norm}[1]{\lVert#1\rVert}

\newcommand{\om}{\omega}
\newcommand{\db}{\overline\partial}
\newtheorem{thm}{Theorem}[section]
\newtheorem{lemma}[thm]{Lemma}
\newtheorem{prop}[thm]{Proposition}
\newtheorem{cor}[thm]{Corollary}
\theoremstyle{definition}
\newtheorem{rem}[thm]{Remark}
\theoremstyle{definition}
\newtheorem{defn}[thm]{Definition}

\newcommand{\comment}[1]{}

\begin{document}

\title{Exponential Estimate for the asymptotics of Bergman kernels}
\date{\today}

\author{Xiaonan Ma}
\address{Institut Universitaire de France
\&Universit\'e Paris Diderot - Paris 7,
UFR de Math\'ematiques, Case 7012,
75205 Paris Cedex 13, France}
\email{ma@math.jussieu.fr}
\thanks{Partially supported by Institut Universitaire de France}
\author{George Marinescu}
\address{Universit{\"a}t zu K{\"o}ln,  
Mathematisches Institut, Weyertal 86-90, 50931 K{\"o}ln, Germany\\
    \& Institute of Mathematics `Simion Stoilow', Romanian Academy,
Bucharest, Romania}
\email{gmarines@math.uni-koeln.de}
\thanks{Partially supported by DFG funded projects SFB/TR 12 
and MA 2469/2-1}

\begin{abstract}
We prove an exponential estimate for 
the asymptotics of Bergman kernels of a positive line 
bundle under hypotheses of bounded geometry.
We give further Bergman kernel proofs of complex geometry results, 
such as separation of points, existence of local coordinates and 
holomorphic convexity by sections of positive line bundles.
\end{abstract}

\maketitle

\setcounter{section}{-1}
\section{Introduction} \label{s1}

Let $(X,\om)$ be symplectic manifold of real dimension
$2n$. Assume that there exists a Hermitian line bundle $(L,h^L)$ 
over $X$ endowed with a Hermitian connection $\nabla^L$ 
with the property that 
\begin{align}\label{eq:0.1}
R^L=   -2\pi\sqrt{-1}\, \omega,
\end{align}
 where
$R^L=(\nabla^L)^2$ is the curvature of $(L,\nabla^L)$. Let
$(E,h^E)$ be a Hermitian vector bundle  on $X$ with Hermitian
connection $\nabla^E$ and its curvature $R^E$.

Let $J$ be an almost complex structure which is
compatible with $\om$ (i.e., $\om$ is $J$-invariant and 
$\om(\cdot, J\cdot)$ defines a metric on $TX$).
Let $g^{TX}$ be a  $J$-invariant Riemannian metric on $X$.
Let $d(x,y)$ be the Riemannian distance on $(X,g^{TX})$.

The \spin Dirac operator $D_p$ acts on
$\Omega^{0,{\scriptscriptstyle{\bullet}}}(X,L^p\otimes E)
=\bigoplus_{q=0}^n\Omega^{0,q}(X,L^p\otimes E)$, 
the direct sum of
spaces of $(0,q)$--forms with values in $L^p\otimes E$.

We refer to the orthogonal projection $P_p$ from 
$L^2(X,E_p)$, the  space of $L^2$-sections of 
$E_p:=\Lambda^{\bullet}(T^{\ast\,(0,1)}X)\otimes L^p\otimes E$, 
onto $\Ker (D_p)$ as the {\emph{Bergman projection}} of $D_p$\,. 
The Schwartz kernel $P_p(\cdot,\cdot)$ of $P_p$
with respect to the Riemannian volume form $dv_X(x')$
of $(X,g^{TX})$
 is called the {\emph{Bergman kernel}} of $D_p$. 

\begin{thm}\label{t0.1}
Suppose that $(X, g^{TX})$ is complete  and 
$R^L, R^{E}$, $J$, $g^{TX}$
have bounded geometry {\rm(}i.e., they and their derivatives of 
any order are uniformly bounded on $X$ in the norm induced by 
$g^{TX}$, $h^{E}$, and the injectivity radius of 
$(X,g^{TX})$ is positive{\rm)}.
Assume also that there exists $\varepsilon>0$ such that on $X$,
\begin{align}\label{eq:0.6}
\sqrt{-1}R^L(\cdot, J\cdot)> \varepsilon g^{TX}(\cdot,\cdot)\,.
\end{align}
Then there exist $\boldsymbol{c} >0$, $\boldsymbol{p}_{0}>0$, 
which can be determined explicitly
from the geometric data {\rm(}cf.\ \eqref{bk3.30}{\rm)} such that  
for any $k\in \N$, there exists $C_k>0$ such that 
for any $p\geqslant \boldsymbol{p}_{0}$, $x,x'\in X$, we have 
\begin{align}\label{eq:0.7}
\left| P_p(x,x')\right|_{\cC^k} \leqslant C_k \, p^{n+\frac{k}{2}}
\, \exp\!\left(- \boldsymbol{c} \,\sqrt{p}\, d(x,x')\right).
\end{align}
\end{thm}
The pointwise $\cC^k$-seminorm $\left|S(x,x')\right|_{\cC^k}$ 
of a section $S\in\cC^\infty(X\times X,E_p\boxtimes E_p^*)$ 
at a point $(x,x')\in X\times X$ is the sum of the norms
induced by $h^L$, $h^E$ and $g^{TX}$ of the derivatives 
up to order $k$ of $S$ with respect to
the connection induced by $\nabla^L$, $\nabla^E$
and the Levi-Civita connection $\nabla^{TX}$ evaluated at $(x,x')$.

Assume now $X=\C^n$ with the standard  trivial metric,
$E=\C$ with trivial metric.
Assume also $L=\C$ and $h^L=e^{-\varphi}$ where 
$\varphi:X\to\R$ is a smooth plurisubharmonic potential 
such that \eqref{eq:0.6} holds.
Then the estimate \eqref{eq:0.7}  with $k=0$ was basically 
obtained by  \cite{Christ91} for $n=1$, 
\cite{Delin98}, \cite{Lindholm01} for $n\geqslant 1$ (cf. also 
\cite{Berndtsson03}).
In \cite[Theorem 4.18]{DLM06} (cf. \cite[Theorem 4.2.9]{MM07}),
a refined version of \eqref{eq:0.7}, i.e., the asymptotic expansion
of $P_p(x,x')$ for $p\to +\infty$ 
with the exponential estimate was obtained.

When $(X,J,\om)$ is a compact K\"ahler manifold, 
$E=\C$ with trivial metric,
$g^{TX}= \om(\cdot,J\cdot)$ and (\ref{eq:0.1}) holds,
a better estimate than \eqref{eq:0.7}
with $k=0$ and $d(x,x')>\delta>0$ was obtained in \cite{Christ13}.

Recently,  Lu and Zelditch announced in  
\cite[Theorem 2.1]{LuZelditch13} 
the estimate  \eqref{eq:0.7}  with $k=0$ and $d(x,x')>\delta>0$
when $(X,J,\om)$ is a complete K\"ahler manifold, 
$E=\C$ with trivial metric,
$g^{TX}= \om(\cdot,J\cdot) $ and (\ref{eq:0.1}) holds.
It is surprising that \cite[Theorem 2.1]{LuZelditch13} requires 
no other condition on the curvatures (cf.\ Remarks \ref{t4.2}, \ref{r4.1}). 

Theorem \ref{t0.1} was known to the authors for several years, 
being an adaptation of \cite[Theorem 4.18]{DLM06}.
The recent papers \cite{Christ13,LuZelditch13} motivated us
to publish our proof.

The next result describes the relation between 
the Bergman kernel on a Galois covering
of a compact symplectic manifold and the Bergman kernel on
the base.
\begin{thm}\label{t4.15}
Let $(X,\om)$ be a compact symplectic manifold.
Let $(L,\nabla^L,h^L)$, $(E,\nabla^E,h^E)$, $J$, $g^{TX}$ 
be given as above.
Consider a Galois covering $\pi:\wi{X}\to X$ and let $\Gamma$ 
be the group of deck transformations. Let us decorate 
with tildes the preimages of objects
living on the quotient, e.\,g., $\wi{L}=\pi^*L$, 
$\wi{\om}=\pi^*\om$ etc. Let $\wi{D}_p$ be the 
\spin Dirac operator acting on
$\Omega^{0,{\scriptscriptstyle{\bullet}}}
(\wi{X},\wi{L}^p\otimes\wi{E})$ and
$\wi{P}_p$ be the orthogonal projection  from 
$L^2(\wi{X},\wi{E}_p)$ onto $\Ker(\wi{D}_p)$.
There exist $\boldsymbol{p}_1$
which depends only on the geometric data on $X$, such that
for any $p> \boldsymbol{p}_1$
we have
\begin{equation}\label{eq:5.18}
\sum_{\gamma\in\Gamma}\wi{P}_p(\gamma x,y)
=P_p(\pi(x),\pi(y))\,,\:\:\text{for any $x,y\in\wi{X}$}.
\end{equation}
\end{thm}
The formula (\ref{eq:5.18}) in the compact K\"ahler situation 
and for $E=\C$ also appeared 
in \cite{LuZelditch13} by a different method.

We refer to \cite{Ma10,MM11} for further applications of the 
off-diagonal expansion of the Bergman kernel to 
the Berezin-Toeplitz quantization.

This paper is organized as follows: In Section \ref{s2}, we explain 
the spectral gap property of Dirac operators
and the elliptic estimate. In Section \ref{s3}, we establish 
Theorems \ref{t0.1}, \ref{t4.15}. In Section \ref{s4}, 
we show what the result becomes in the complex case. 
We give some applications of Theorem \ref{t0.1}
and of the diagonal expansion of the Bergman kernel.

\section{Dirac operator and elliptic estimates}\label{s2}

The almost complex structure $J$ induces a splitting
$T X\otimes_\R \C=T^{(1,0)}X\oplus T^{(0,1)}X$,
where $T^{(1,0)}X$ and $T^{(0,1)}X$
are the eigenbundles of $J$ corresponding to the eigenvalues 
$\sqrt{-1}$
and $-\sqrt{-1}$ respectively. Let $T^{*(1,0)}X$  and $T^{*(0,1)}X$
be the corresponding dual bundles.
For any $v\in TX$ with decomposition $v=v_{1,0}+v_{0,1}
\in T^{(1,0)}X\oplus T^{(0,1)}X$,  let 
${ v^\ast_{1,0}}\in T^{*(0,1)}X$ be the metric
dual of $v_{1,0}$. Then $c(v)=\sqrt{2}({ v^\ast_{1,0}}\wedge-
i_{v_{\,0,1}})$ defines the Clifford action of $v$ on
$\Lambda (T^{*(0,1)}X)$, where $\wedge$ and $i$
denote the exterior and interior product respectively. We denote
\begin{equation}\label{main2.1a}
    \Lambda^{0,\scriptscriptstyle\bullet}=\Lambda^{\bullet}
(T^{\ast\,(0,1)}X),\quad
E_p:=\Lambda^{0,\scriptscriptstyle{\bullet}}\otimes L^p\otimes E.
\end{equation}
Along the fibers of $E_p$,
we consider the pointwise Hermitian product 
$\langle\cdot,\cdot\rangle$  
induced by $g^{TX}$, $h^L$ and $h^E$. 
Let $dv_X$ be the Riemannian volume form of $(TX, g^{TX})$.
The $L^2$--Hermitian product on the space 
$\Omega^{0,\bullet}_0(X,L^p\otimes{E})$ 
of smooth compactly supported sections of $E_p$ is given by 
\begin{equation}\label{l2}
\langle s_1,s_2\rangle=\int_X\langle s_1(x),s_2(x)\rangle\,
dv_X(x)\,.
\end{equation}
We denote the corresponding norm with $\norm{\cdot}_{L^2}$ and 
with $L^2(X,E_p)$ the completion of 
$\Omega^{0,\scriptscriptstyle{\bullet}}_{0}(X,L^p\otimes E)$ 
with respect to this norm,
and $\|B\|^{0,0}$ the norm of
$B\in \cL(L^2(X, E_{p}))$ with respect to $\norm{\cdot}_{L^2}$.

Let $\nabla^{TX}$ be the Levi-Civita connection
of the metric $g^{TX}$, and let $\nabla^{\det_{1}}$ be the
connection on  $\det(T^{(1,0)}X)$ induced by $\nabla^{TX}$
by projection.
By \cite[\S 1.3.1]{MM07}, 
$\nabla^{TX}$ (and  $\nabla^{\det_{1}}$) induces canonically
a Clifford connection $\nabla^{\text{Cliff}}$ on 
$\Lambda (T^{*(0,1)}X)$.
Let $\nabla^{E_p}$ be the connection on
$E_p=\Lambda (T^{*(0,1)}X)\otimes L^p\otimes E$
induced by $\nabla^{\text{Cliff}}$, $\nabla^L$ and  $\nabla^E$.

We denote by $\spec(B)$ the spectrum of an operator $B$.

\begin{defn}\label{Dirac}
The \spin Dirac operator $D_p$ is defined by
\begin{equation}\label{defDirac}
D_p=\sum_{j=1}^{2n}c(e_j)\nabla^{E_p}_{e_j}:
\Omega^{0,\scriptscriptstyle{\bullet}}(X,L^p\otimes E)
\longrightarrow
\Omega^{0,\scriptscriptstyle{\bullet}}(X,L^p\otimes E)\,,
\end{equation}
with $\{e_i\}_i$  an orthonormal basis of $TX$.
\end{defn}
\noindent
Note that $D_p$ is a formally self--adjoint, first order 
elliptic differential operator.
Since we are working on a complete manifold $(X,g^{TX})$, 
$D_p$ is essentially self--adjoint.
This follows e.\,g.,\ from an easy modification of the 
Andreotti-Vesentini lemma \cite[Lemma 3.3.1]{MM07} 
(where the particular case of a complex manifold and 
the Dirac operator \eqref{eq4.1} is considered).  
Let us denote by $D_p$ the self-adjoint extension of $D_p$ 
defined initially on the space of smooth compactly supported forms, 
and by $\Dom(D_p)$ its domain.
By the proof of \cite[Theorems 1.1, 2.5]{MM02}
(cf. \cite[Theorems 1.5.7, 1.5.8]{MM07}), we have :

\begin{lemma}\label{t2.1}
    If $(X,g^{TX})$ is complete and $\nabla^{TX} J$, $R^{TX}$
    and $R^{E}$ are uniformly bounded on $(X,g^{TX})$,
    and (\ref{eq:0.6}) holds, then 
there exists $C_L>0$ such that for any $p\in \N$, 
and any 
$s\in\bigoplus_{q\geqslant 1}\Omega_0^{0,q}(X,L^p\otimes E)$,
we have 
\begin{equation}\label{main1}
\norm{D_{p}s}^2_{L^ 2}\geqslant\big(2p\mu_0-C_L\big)
\norm{s}^2_{L^ 2}\,,
\end{equation}
 with
\begin{equation}\label{b1}
\mu_0=\displaystyle\inf_{{u\in T_x^{(1,0)}X,\,x\in X}} 
R^L_x(u,\overline{u})\big\slash|u|^2_{g^{TX}} >0.
\end{equation}
Moreover  
\begin{align} \label{main1.1}
{\spec} (D^2_p )\subset \lbrace0\rbrace\cup 
[2p\mu_0 -C_L,+\infty[\,.
\end{align}
\end{lemma}

From now on, we assume that $(X,g^{TX})$ is complete and 
$R^L, R^{E}$,  $J$, $g^{TX}$ have bounded geometry.

The following elliptic estimate will be applied to get 
the kernel estimates.
\begin{lemma}\label{t2.2} Given a sequence of smooth forms 
    $s_p\in\bigcap_{\ell\in\N}\Dom(D_p^\ell)\subset L^2(X,E_p)$ 
    and a sequence $C_p>0$ {\rm(}$p\in\N${\rm)}, 
    assume that for any $\ell\in\N$, there exists $C_\ell'>0$ 
 such that for any $p\in\N^*$, 
 \begin{align}\label{main2.1}
\left\|  \left(\frac{1}{\sqrt{p}}\,D_{p}\right)^{\!\ell} s_p 
\right\|_{L^2}  \leqslant C_\ell' \, C_p\,.
\end{align}
Then  for any $k\in \N$, there exists $C_k>0$ such that 
for any $p\in \N^*$ and $x\in X$ 
the pointwise $\cC^k$-seminorm satisfies
 \begin{align}\label{main2.2}
\big|s_p\big|_{\cC^k}(x) \leqslant 
C_k \, C_p\, p^{\frac{n+k}{2}}.
\end{align}
\end{lemma}
\begin{proof} Let $\inj^X$ be the injectivity radius of 
    $(X, g^{TX})$, and let $\var\!\in\,]0, \inj^X[$\,.
We denote by $B^{X}(x,{\varepsilon})$ and  
$B^{T_xX}(0, \varepsilon)$ the open balls
in $X$ and $T_x X$ with center $x\in X$ and radius 
$ \varepsilon$, respectively.
The exponential map $ T_x X\ni Z \mapsto \exp^X_x(Z)\in X$
is a diffeomorphism from $B^{T_xX} (0, \varepsilon)$  
on $B^{X} (x, \varepsilon)$.  
From now on, we identify $B^{T_xX}(0, \varepsilon)$ with 
$B^{X}(x, \varepsilon)$ for $\varepsilon < \inj^X$.

For $x_0\in X$, we work in the normal coordinates on
$B^{X} (x_0, \varepsilon)$.
 We identify $E_Z$, $L_Z$, $\Lambda (T^{*(0,1)}_Z X)$
 to  $E_{x_0}$, $L_{x_0}$, $\Lambda (T^{*(0,1)}_{x_0}X)$
 by parallel transport with respect
to the connections $\nabla ^E$, $\nabla ^L$,
$\nabla ^{\text{Cliff}}$ along the curve
$[0,1]\ni u \mapsto uZ$. Thus on $B^{T_{x_0}X}(0, \varepsilon)$,
$(E_{p}, h^{E_{p}})$ is identified to the trivial Hermitian bundle 
$(E_{p,x_{0}}, h^{E_{p, x_{0}}})$.
 Let $\{e_i\}_i$
be an orthonormal basis of $T_{x_0}X$.
Denote by  $\nabla_U$  the ordinary differentiation
 operator on $T_{x_0}X$ in the direction $U$.
Let $\Gamma^E, \Gamma^L, \Gamma^{\Lambda ^{0,\bullet}}$
be the corresponding connection forms of $\nabla^E$, $\nabla^L$
and $\nabla^{\text{Cliff}}$ with respect to any fixed frame for
$E,L$, $\Lambda (T^{*(0,1)}X)$ which is parallel along the curve
$[0,1]\ni u \mapsto uZ$
under the trivialization on $B^{T_{x_0}X}(0, \varepsilon)$.

Let $\{\wi{e}_i\}_i$ be an orthonormal frame on $TX$.
On $B^{T_{x_0}X}(0, \varepsilon)$, we have
\begin{align}\label{lm4.14}
\begin{split}
& \nabla^{E_{p,x_0}}_{\textstyle\cdot}= \nabla_{\textstyle\cdot} 
+p\, \Gamma ^L(\cdot)
+\Gamma ^{\Lambda ^{0,\bullet}}(\cdot)+ \Gamma ^E(\cdot),\\
&D_p = c(\wi{e}_j) \Big(\nabla_{\wi{e}_j} +p \Gamma ^L(\wi{e}_j)
+\Gamma ^{\Lambda ^{0,\bullet}}(\wi{e}_j)
+ \Gamma ^E(\wi{e}_j) \Big ).
\end{split}
\end{align}
By  \cite[Lemma 1.2.4]{MM07},  for $\Gamma^{\bullet} = 
\Gamma^E, \Gamma^L, \Gamma^{\Lambda ^{0,\bullet}}$, we have
\be\label{lm01.18}
\Gamma^{\bullet}_{Z}(e_j) = \frac{1}{2} R^{\bullet}_{x_0}(Z, e_j) 
+ \cO(|Z|^2).
\ee
Using an unit vector $S_L$ of $L_{x_0}$, we get an isometry 
$E_{p,x_0}\simeq (\Lambda (T^{*(0,1)}X)\otimes E)_{x_0}
=: \bE_{x_0}\,.$
\comment{As the operator $L_{p,x_0}$ takes values in
$\End(E_{p,x_0})= \End(\bE)_{x_0}$ 
(using the natural identification
$\End(L^p)\simeq \C$, which does not depend on $S_L$),
thus our formulas  do not depend the choice of $S_L$.
Now, under this identification, we will consider $L_{p,x_0}$
acting on $\cC^\infty (X_0,\bE_{x_0})$.
}
For  $s \in \cC^{\infty}(\R^{2n}, \bE_{x_0})$, $Z\in \R^{2n}$ and
$t=\frac{1}{\sqrt{p}}$\,, set\index{$L^t_2$}
\begin{align}\label{lm4.28}
\begin{split}
&(S_{t} s ) (Z) =s (Z/t), \quad
\nabla_{t}=  S_t^{-1}
t \nabla ^{E_{p,x_0}} S_t,\\
&\boldsymbol{D}_t =  S_t^{-1} t D_p  S_t\,.
\end{split}
\end{align}
Let $f : \R \to [0,1]$ be a smooth even function such that
 \begin{align} \label{lm4.19}
f(v) = \left \{ \begin{array}{ll}  1 \quad {\rm for}
\quad |v| \leqslant  \var/2, \\
  0 \quad {\rm for} \quad |v| \geqslant  \var.
\end{array}\right.
\end{align}
Set 
\begin{align}\label{lm4.19a}
\sigma_p (Z)= s_{p} (Z) f(|Z|/t), \quad  \wi{\sigma}_p = 
S_{t}^{-1}\sigma_p \in  \cC^{\infty}_{0}(\R^{2n}, \bE_{x_0}).
\end{align}
Then by (\ref{lm4.28}) and (\ref{lm4.19a}), we have
\begin{equation}\label{lm4.20a}\begin{split}
  &\wi{\sigma}_p(Z) = s_{p} ( t \,Z) f(|Z|)\,,\\
 & (\boldsymbol{D}_t ^l  \wi{\sigma}_p)(Z)= \sum_{j=0}^{l}
 \begin{pmatrix} l \\ j \end{pmatrix}
\underbrace{\big[D_{p},\dots,[D_{p}}_{\text{$(l-j)$ times}}, 
f(|\cdot|)]\ldots\big](Z)  \, 
(\boldsymbol{D}_t ^j S_{t}^{-1}s_{p})(Z)\,.
\end{split}\end{equation}
By \eqref{lm4.14}, $\big[D_{p},\ldots,[D_{p}, f(|\cdot|)]\ldots\big]$
is uniformly bounded on $B^{T_{x_0}X}(0, \varepsilon)$, 
with respect to $x_{0}\in X$, $p\in \N$, and 
$(\boldsymbol{D}_t ^j S_{t}^{-1}s_{p})(Z)
= \left(\Big(\frac{1}{\sqrt{p}}D_{p}\Big)^{j} s_p \right)(tZ)$.
Thus (\ref{main2.1}), (\ref{lm4.14}) and (\ref{lm4.20a}) 
imply
\begin{equation}\label{lm4.21a}
\begin{split}
\int\limits_{|Z|<\var} | \boldsymbol{D}_t ^l  \wi{\sigma}_p(Z)|^2\, dZ 
&\leqslant C_{0} p^n \int\limits_{|Z|<\var/\sqrt{p}}
\sum_{j=0}^{l}\left|\Big(\frac{1}{\sqrt{p}}D_{p}\Big)^{j} s_p (Z)
\right|^2\, dZ\\
&\leqslant  C_{1} \sum_{j=0}^{l}(C_j'C_p)^2 p^{n}.
\end{split}
\end{equation}
Now by (\ref{lm4.14}), (\ref{lm01.18}) and  (\ref{lm4.28}), we get
\begin{align}\label{lm4.18a}
\begin{split}
&\nabla_{t}= \nabla_{0}+ \mO(t), \quad  
\boldsymbol{D}_t = \boldsymbol{D}_{0}+ \mO(t),
\end{split}
\end{align}
where $\boldsymbol{D}_{0}$ is an elliptic operator on $\R^{2n}$.
On $\R^{2n}$, we use the usual $\cC^k$-seminorm,
i.\,e., $|f|_{\cC^k}(x):= \sum_{l\leqslant k} 
|\nabla_{e_{i_{1}}}\cdots \nabla_{e_{i_{l}}} f|(x)$.
By the Sobolev embedding theorem (cf.\ \cite[\!Theorem A.1.6]{MM07}),
(\ref{lm4.21a}), (\ref{lm4.18a}) 
and our assumption on bounded geometry, we obtain that
for any $k\in \N$, there exists $C_{k}>0$ such that 
for any $x_{0}\in X$, $p\in \N^*$, we have
\begin{align}\label{lm4.22a}
 \big| \wi{\sigma}_p\big|^2_{\cC^k} (0)   
 \leqslant  C_k\, C_p\, p^{n/2}\,.
\end{align}
Going back to the original coordinates (before rescaling), we get
\begin{align}\label{lm4.23a}
 \big| \sigma_p\big|^2_{\cC^k} (x_{0})\leqslant  C_k \, C_p\,
 p^{\frac{n+k}{2}}\,.
\end{align}
The proof of Lemma \ref{t2.2} is completed.
\end{proof}

\section{Proofs of Theorems \ref{t0.1} and \ref{t4.15}}\label{s3}

For $x,x^\prime\in X$ let 
$\exp\!\Big(\!-\frac{u}{p}D_{p}^{2}\Big)(x,x^\prime)$, 
$\left(\frac{1}{p}D_{p}^{2}\exp\!\Big(\!-\frac{u}{p}D_{p}^{2}\Big)
\right)(x,x^\prime)$  be the smooth kernels 
of the operators $\exp\!\Big(\!-\frac{u}{p}D_{p}^{2}\Big)$, 
$\frac{1}{p}D_{p}^{2}\exp\!\Big(\!-\frac{u}{p}D_{p}^{2}\Big)$
with respect to $dv_{X}(x^\prime)$.

\begin{thm}\label{bkt3.3} There exists $\boldsymbol{a} >0$ 
 such that for any $m\in \N$, $u_0>0$, there exists $C>0$ 
 such that for $u\geqslant u_0$, $p\in\N^*$,
 $x,x^\prime\in X$, we have 
 \begin{align}\label{bk3.5}
\begin{split}
&\Big |\exp\!\Big(\!-\frac{u}{p}D_{p}^{2}\Big)(x,x^\prime)
\Big |_{\cC^{m}}   \leqslant C p^{n+\frac{m}{2}}
\exp\!\Big(\mu_0 u- \frac{\boldsymbol{a} p}{u} d(x,x')^2\Big),
\end{split}
\end{align}
and that for  $u\geqslant u_0$, $p\geqslant 2 C_{L}/\mu_{0}$, 
$x,x^\prime\in X$, we have 
 \begin{align}\label{bk3.5a}
\begin{split}
&\Big |\left(\frac{1}{p}\,D_{p}^{2}
\exp\!\Big(\!-\frac{u}{p}D_{p}^{2}\Big)\right)(x,x^\prime)
\Big |_{\cC^{m}}   \leqslant C p^{n+\frac{m}{2}}
\exp\!\Big( -\frac{1}{2}\mu_0 u
- \frac{\boldsymbol{a} p}{u} d(x,x')^2\Big). 
\end{split}
\end{align}
\end{thm}
\begin{proof} For any $u_0>0$, $k\in \N$, 
there exists $C_{u_0,k}>0$ such that for 
$u\geqslant u_0$,  $p\in\N^*$, we have 
     \begin{align}\label{bk3.6}
\begin{split}
&\Big \| \left(\frac{1}{\sqrt{p}}\,D_{p}\right)^{\!k} 
\exp\!\Big(\!-\frac{u}{p}D_{p}^{2}\Big)\Big \|^{0,0}   
\leqslant C_{u_0,k},
\end{split}
\end{align}
and by  Lemma \ref{t2.1}, that for  
$u\geqslant u_0$, $p\geqslant 2 C_{L}/\mu_{0}$,  we have 
 \begin{align}\label{bk3.6a}
\begin{split}
&\Big \| \left(\frac{1}{\sqrt{p}}\,D_{p}\right)^{k} 
\frac{1}{p}\,D_{p}^{2} \exp\!\Big(\!-\frac{u}{p}D_{p}^{2}\Big)
\Big \|^{0,0}   \leqslant C_{u_0,k} e^{- \mu_0 u}.   
\end{split}
\end{align}
From Lemma \ref{t2.2},  (\ref{bk3.6}) and (\ref{bk3.6a}),  
for any $u_0>0$, $m\in \N$, there exists $C_{u_0,m}'>0$
such that for $u\geqslant u_0$, $p\in\N^*$, 
$x,x'\in X$, we have
  \begin{align}\label{bk3.7}
\begin{split}
&\Big |\exp\!\Big(\!-\frac{u}{p}D_{p}^{2}\Big)(x,x^\prime)
\Big |_{\cC^{m}}   \leqslant C_{u_0,m}' \, p^{n+\frac{m}{2}},
\end{split}
\end{align}
and that for  $u\geqslant u_0$, $p\geqslant 2 C_{L}/\mu_{0}$, 
$x,x'\in X$,  we have 
 \begin{align}\label{bk3.7a}
\begin{split}
&\Big |\left(\frac{1}{p}\,D_{p}^{2}
\exp\!\Big(\!-\frac{u}{p}D_{p}^{2}\Big)\right)(x,x^\prime)
\Big |_{\cC^{m}}   \leqslant C_{u_0,m}' \,  
p^{n+\frac{m}{2}} e^{- \mu_0 u} . 
\end{split}
\end{align}
        
To obtain (\ref{bk3.5}) and (\ref{bk3.5a})  in general, we proceed as in the proof of 
\cite[Theorem 4.11]{DLM06} and \cite[Theorem 11.14]{B95} 
(cf. \cite[Theorem 4.2.5]{MM07} ).
For $h>1$ and $f$ from \eqref{lm4.19},  put
\begin{align}  \label{bk3.10}\begin{split}
&K_{u,h}(a)= \int_{-\infty}^{+\infty} \cos(i v\sqrt{2u} a)
 \exp\Big(-\frac{v^2}{2}\Big)
\Big (1-f \big(\frac{1}{h}\sqrt{2u} v\big) \Big ) 
\frac{dv}{\sqrt{2\pi}}\,,\\
&H_{u,h}(a)= \int_{-\infty}^{+\infty} \cos(i v\sqrt{2u} a)
 \exp\Big(-\frac{v^2}{2}\Big)
f \big(\frac{1}{h}\sqrt{2u} v\big)  \frac{dv}{\sqrt{2\pi}}\,\cdot
\end{split}\end{align}
By (\ref{bk3.10}), we infer
\begin{align}  \label{bk3.10a}
&K_{u,h}\left(\frac{1}{\sqrt{p}}D_{p}\right) + 
H_{u,\,h}\left(\frac{1}{\sqrt{p}}D_{p}\right) 
= \exp\!\Big(\!-\frac{u}{p}D_{p}^{2}\Big).
\end{align}
Using finite propagation speed \index{finite propagation speed}
 of solutions of hyperbolic equations \cite[Theorem D.2.1]{MM07},
 and (\ref{bk3.10}), we find that 
\begin{equation}\label{bk3.11a}
\begin{split}
&\supp H_{u,h}\left(\frac{1}{\sqrt{p}}D_{p}\right) (x,\cdot) 
\subset B^X(x, \var  h/\sqrt{p}), \text{  and   }\\
&H_{u,h}\left(\frac{1}{\sqrt{p}}D_{p}\right) (x,\cdot)
\text{  depends only  on the restriction of }
D_{p} \text{ to }  B^X(x, \var  h/\sqrt{p}).
\end{split}
\end{equation}
Thus from (\ref{bk3.10a}) and (\ref{bk3.11a}),  
we get  for $x,x'\in X$,
\begin{align}\label{bk3.13}
K_{u,\,h} \left(\frac{1}{\sqrt{p}}\,D_{p}\right) (x,x')
=\exp\!\Big(\!-\frac{u}{p}D_{p}^{2}\Big)(x,x^\prime),
\quad {\rm if} \,\,
\sqrt{p} \,d(x,x')\geqslant \var  h.
\end{align}
By (\ref{bk3.10}),  there exist $C',C_1>0$ such that for any
$c>0$, $m\in \N$,
there is $C>0$ such that for $u\geqslant u_0$, $h>1$,
$a\in \C, |{\rm Im} (a)|\leqslant c$,  we have
\begin{align} \label{bk3.11}
|a|^m |K_{u,h}(a)| 
\leqslant C \exp \Big( C'c^2 u- \frac{C_1}{u} h^2\Big).
\end{align}
Using Lemma \ref{t2.2} and \eqref{bk3.11} 
we find that for $\boldsymbol{K}(a)= K_{u,h}(a)$
or $a^2K_{u,h}(a)$,
 \begin{align}  \label{bk3.15}
\Big |\boldsymbol{K}\left(\frac{1}{\sqrt{p}}\,D_{p}\right) (x,x')
\Big |_{\cC^{m}}   \leqslant    C_{2}  p^{n+\frac{m}{2}}  
\exp \left(C'c^2 u- \frac{C_1}{u} h^2\right).
\end{align}
Setting $h=\sqrt{p}\,d(x,x')/\var$ in (\ref{bk3.15}), we get that 
for any $x,x'\in X$, $p\in \N^*$, $u\geqslant u_0$,
such that $\sqrt{p}d(x,x')\geqslant \var$, we have
\begin{align}  \label{bk3.16}
\Big |\boldsymbol{K}\left(\frac{1}{\sqrt{p}}\,D_{p}\right) (x,x')
\Big |_{\cC^{m}}\leqslant C  p^{n+\frac{m}{2}}
\exp \left(C'c^2 u- \frac{C_1}{\var^2u} p \, d(x,x')^2\right).
 \end{align}
 By (\ref{bk3.7}), (\ref{bk3.13}) and (\ref{bk3.16}),
we infer (\ref{bk3.5}).
By (\ref{bk3.7a}), (\ref{bk3.13}) and (\ref{bk3.16}),
we infer (\ref{bk3.5a}).
The proof of Theorem \ref{bkt3.3} is completed.
\end{proof}
\begin{proof}[Proof of Theorem \ref{t0.1}]
Analogue to \cite[(4.89)]{DLM06} (or \cite[(4.2.22)]{MM07}), 
we have
\begin{align}\label{bk3.20}
\exp\!\Big(\!-\frac{u}{p}\,D_{p}^{2}\,\Big)- P_p
= \int_u^{+\infty} \frac{1}{p}\,D_{p}^{2}\,
\exp\!\Big(\!-\frac{u_1}{p}\,D_{p}^{2}\,\Big) du_1.
\end{align}
 Note that 
 $\frac{1}{4}\mu_0 u+\frac{\boldsymbol{a}}{u} p \, d(x,x')^2
 \geqslant \sqrt{\boldsymbol{a}\mu_0 p}\,  d(x,x')$, thus
\begin{equation}\label{bk3.22}
\begin{split}
\int_u^{+\infty} &\exp\!\Big(\!-\frac{1}{2}\,\mu_0 u_1
-\frac{\boldsymbol{a}}{u_1} p \,d(x,x')^2\Big)du_1\\
&\leqslant \exp(-\sqrt{\boldsymbol{a}\mu_0 p} \, d(x,x'))
\int_u^{+\infty}\exp\!\Big(\!-\frac{1}{4}\,\mu_0 u_1\Big)du_1 \\
& = \frac{4}{\mu_0}\exp\!\left(-\frac{1}{4}\mu_0 u
 - \sqrt{\boldsymbol{a}\mu_0 p }\,  d(x,x')\right).
 \end{split}
\end{equation}
By 
 \eqref{bk3.5a},  \eqref{bk3.20} 
and \eqref{bk3.22},  there exists $C>0$ such that for 
$u\geqslant u_{0}$, $p\geqslant 2 C_{L}/\mu_{0}$, 
$x,x'\in X$, we have
\begin{align}\label{bk3.22b}
\begin{split}
& \left|\left(\exp\!\Big(\!-\frac{u}{p}\,D_{p}^{2}\Big)
- P_p\right)\!(x,x^\prime)\right|_{\cC^{m}}   
\leqslant C p^{n+\frac{m}{2}}
\exp\!\left( -\frac{1}{4}\mu_0 u- \sqrt{\boldsymbol{a}\mu_0 p }\,  
d(x,x')\right). 
\end{split}
\end{align}
By 
\eqref{bk3.5} and \eqref{bk3.22b}, 
we get (\ref{eq:0.7}) with 
\begin{equation}\label{bk3.30}
\boldsymbol{p}_{0}=2 C_{L}/\mu_{0}\, ,
\quad \boldsymbol{c}= \sqrt{\boldsymbol{a}\mu_0}.
\end{equation}
The proof of Theorem \ref{t0.1} is completed.
\end{proof}
\begin{rem}\label{bkt3.4a}  Let $A\in \Omega^3(X)$. We assume 
 $A$ and its derivatives are bounded on $X$. Set 
\begin{align}\label{bk3.24b}
{^cA} = \sum_{i<j<k} A(e_i,e_j,e_k) c(e_i)c(e_j)c(e_k), 
\quad  D^{A}_p:=D_p + {^cA}.
\end{align}
Then $ D^{A}_p$ is a modified Dirac operator 
(cf. \cite[\S 1.3.3]{MM07}, \cite{B89a}).
As we are in the bounded geometry context, 
Lemmas \ref{t2.1}, \ref{t2.2} still hold if we replace $D_p$ 
by $D_p^A$. This implies that Theorem  \ref{bkt3.3} holds for 
$D_p^A$, thus Theorem  \ref{t0.1} holds for the orthogonal 
projection from from $L^2(X,E_p)$ onto $\Ker (D_p^A)$.
\end{rem}

\begin{proof}[Proof of Theorem \ref{t4.15}]
    It is standard that for any $p\in \N^*$, $u>0$,
  and $x,y\in\wi{X}$,   
\begin{align}\label{eq:5.21}
   \exp\!\Big(\!-uD_{p}^{2}\Big)(\pi(x),\pi(y))
 = \sum_{\gamma\in\Gamma}
     \exp\!\Big(\!-u\wi{D}_{p}^{2}\Big)(\gamma x,y)\,.
\end{align}
Let us denote for $r>0$ by 
$N(r)=\#\, B^{\wi{X}}(x,r)\cap\Gamma y$. 
Let $K>0$ be such that the sectional curvature of $(X, g^{TX})$
is $\geqslant -K^2$. 
By 
\cite{Mil68}, there exists $C>0$ such that for any $r>0$, 
$x,y\in \wi{X}$, we have 
\begin{align}\label{eq:5.23}
N(r)\leqslant Ce^{(2n-1)Kr}.
\end{align}
(Note that in the proof of  (\ref{bk3.5}), we 
did not use Lemma \ref{t2.1}.
From (\ref{bk3.5}) and (\ref{eq:5.23}), we know the right hand side 
of (\ref{eq:5.21}) is absolutely convergent, and verifies 
the heat equation on $X$,
thus we also get a proof of (\ref{eq:5.21})).

Take $\boldsymbol{p}_1=\max\{\boldsymbol{p}_0, 
(2n-1)^2K^2/\boldsymbol{c}^2\}$. 
Then for any $p>\boldsymbol{p}_1$, by Theorem \ref{bkt3.3}, 
\eqref{bk3.22} and (\ref{eq:5.23}), we know that
\begin{multline}\label{eq:5.24}
\sum_{\gamma\in\Gamma}
\int_u^{+\infty} \Big(\frac{1}{p}\,\wi{D}_{p}^{2}\,
\exp\!\Big(\!-\frac{u_1}{p}\,\wi{D}_{p}^{2}\,\Big)\Big) 
(\gamma x,y)du_1\\
= \int_u^{+\infty} \sum_{\gamma\in\Gamma}
\Big(\frac{1}{p}\,\wi{D}_{p}^{2}\,
\exp\!\Big(\!-\frac{u_1}{p}\,\wi{D}_{p}^{2}\,\Big)\Big) 
(\gamma x,y)du_1.
\end{multline}
Moreover, \eqref{eq:5.18} and (\ref{eq:5.23}) show that
$\sum_{\gamma\in\Gamma} \wi{P}_p(\gamma x,y)$
is absolutely convergent for any $p>\boldsymbol{p}_1$.
From  Theorem \ref{bkt3.3}, \eqref{bk3.20} for $\wi{D}_{p}$,
\eqref{eq:5.21} and \eqref{eq:5.24},  we get
\begin{multline}\label{eq:5.22}
 \exp\!\Big(\!-\frac{u}{p}\,D_{p}^{2}\Big)(\pi(x),\pi(y))
 -\sum_{\gamma\in\Gamma} \wi{P}_p(\gamma x,y)\\
 = \int_u^{+\infty} \Big(\frac{1}{p}\,D_{p}^{2}\,
\exp\!\Big(\!-\frac{u_1}{p}\,D_{p}^{2}\,\Big)\Big)
(\pi(x),\pi(y)) du_1.
\end{multline}
Now from (\ref{bk3.20}) for $D_{p}$
 and (\ref{eq:5.22}), we obtain \eqref{eq:5.18}.
\end{proof}
\section{The holomorphic case}\label{s4}
We discuss now the particular case of a complex manifold, 
cf.\ the situation of \cite[\S 6.1.1]{MM07}.
Let $(X, J)$ be a complex manifold with complex structure $J$
and complex dimension $n$. 
Let $g^{TX}$ be a Riemannian metric on $TX$ compatible with $J$,
and let $\Theta=g^{TX}(J\cdot,\cdot)$ be the $(1,1)$-form 
associated to $g^{TX}$ and $J$. 
We call $(X,J,g^{TX})$ or $(X,J,\Theta)$ a Hermitian manifold. 
A Hermitian manifold $(X,J,g^{TX})$ is called complete if
$g^{TX}$ is complete.
Moreover let $(L,h^L),(E,h^E)$ be
holomorphic Hermitian vector bundles on $X$ and $\rank (L)=1$.
Consider the holomorphic Hermitian 
(Chern) connections $\nabla^L,\nabla^E$ on $(L,h^L),(E,h^E)$. 

This section is organized as follows. In Section \ref{s4.1},
we explain Theorem \ref{t0.1} in the holomorphic case.
In Section \ref{s4.2}, 
we give some Bergman kernel proofs of 
some known results about separation of points, existence of
local coordinates and holomorphic convexity. 
The usual proofs use the $L^2$ estimates for 
the $\db$-equation introduced by Andreotti-Vesentini and 
H\"ormander. For plenty of informations about 
holomorphic convexity of coverings
(Shafarevich conjecture) and its role 
in algebraic geometry see \cite{Ko95}.

\subsection{Theorem \ref{t0.1} in the holomorphic case}\label{s4.1}

The space of holomorphic sections of $L^p\otimes{E}$ which 
are $L^2$ with respect to the norm given by \eqref{l2} is 
denoted by $H^0_{(2)}(X,L^p\otimes{E})$. 
Let $P_p(x,x')$, $(x,x'\in X)$ be the Schwartz kernel of 
the orthogonal projection $P_p$, from the space of 
$L^2$-sections of $L^p\otimes{E}$ onto 
$H^0_{(2)}(X,L^p\otimes{E})$, with respect to the 
Riemannian volume form $dv_X(x')$ associated to $(X,g^{TX})$. 
\begin{thm}
Let $(X,J,g^{TX})$ be a complete Hermitian manifold.
Assume that 
$R^L, R^{E}$, $J$, $g^{TX}$ have bounded geometry 
and \eqref{eq:0.6} holds.
Then the uniform exponential estimate \eqref{eq:0.7} 
holds for the Bergman kernel $P_p(x,x')$
associated to $H^0_{(2)}(X,L^p\otimes{E})$.
\end{thm}
\begin{proof}
Let $\overline{\partial} ^{L^p\otimes E,*}$  be the formal adjoint 
of the Dolbeault operator $\overline{\partial} ^{L^p\otimes E}$ 
with respect to the Hermitian 
product \eqref{l2} on 
 $\Omega ^{0,\bullet}(X, L^p\otimes E)$.
Set 
 \begin{align} \label{eq4.1}
 D_p = \sqrt{2}\left( \overline{\partial} ^{L^p\otimes E}
+ \overline{\partial} ^{L^p\otimes E,*}\right).
\end{align}
\comment{Then 
$$D_p^2= 2( \overline{\partial} ^{L^p\otimes E}
\overline{\partial} ^{L^p\otimes E,*}
+\overline{\partial} ^{L^p\otimes E,*}\overline{\partial} 
^{L^p\otimes E})$$ 
preserves the $\Z$-grading of 
$\Omega ^{0,\bullet}(X, L^p\otimes E)$.
}
Using the assumption of the bounded geometry, we have 
by \cite[(6.1.8)]{MM07} for $p$ large enough,
\begin{equation} \label{f12}
\Ker (D_p)  =\Ker (D_p^2)  = H^0_{(2)} (X,L^p\otimes E).
\end{equation}
Observe that $D_{p}$ is a modified Dirac operator as in 
Remark \ref{bkt3.4a} with 
$A= \frac{\sqrt{-1}}{4}(\partial -\ov{\partial})\Theta$, see
\cite[~\!Theorem 2.2]{B89a} (cf.\ \cite[~\!Theorem 1.4.5]{MM07}).
In particular, if $(X,J, g^{TX})$ is a complete K\"ahler manifold, 
then the operator $D_{p}$ from \eqref{eq4.1} is a Dirac operator in the sense of Section 
\ref{s3}.

Thus under the assumption of the bounded geometry, 
Theorem \ref{t0.1} still holds for the kernel $P_p(x,x')$ 
in this context.
\end{proof}
\begin{rem}\label{t4.2}
Let $(X,J,g^{TX})$ be a complete Hermitian manifold. 
We assume now that 
\begin{align}\label{eq:4.4}
R^L= -2\pi \sqrt{-1} \Theta.
\end{align}
Then $(X,J,\om)$ is a complete K\"ahler manifold,
$g^{TX}= \om(\cdot,J\cdot) $ and \eqref{eq:0.1} holds.
To get the spectral gap property (Lemma \ref{t2.1}),
or even the H\"ormander $L^2$-estimates, as in 
\cite[Theorem 6.1.1]{MM07}, we need to suppose that 
there exists $C>0$ such that on $X$, 
\begin{align}\label{eq:4.5}
\sqrt{-1}(R^{\det}+R^E)> -C\om \Id_E,
\end{align}
with $R^{\det}$ the curvature of the holomorphic Hermitian 
connection $\nabla^{\det}$ on $K_X^*=\det (T^{(1,0)}X)$.
\end{rem}

\subsection{Holomorphic convexity of manifolds with 
bounded geometry}\label{s4.2} 

We will identify the 2-form $R^L$ with the Hermitian matrix
$\dot{R}^L \in \End(T^{(1,0)}X)$ such that
for $W,Y\in T^{(1,0)}X$,
\begin{equation}\label{lm4.2}
R^L (W,\ov{Y}) = \langle \dot{R}^LW, \ov{Y}\rangle.
\end{equation}

Analogous to the result of Theorem \ref{t0.1} about 
the uniform off-diagonal decay of the Bergman kernel, 
a straightforward adaptation of the technique used in this paper 
yields the uniform diagonal expansion for manifolds and bundles 
with bounded geometry
(cf.\ \cite[~\!Theorem 4.1.1]{MM07} for the compact case, 
\cite[~\!Theorems 6.1.1, 6.1.4]{MM07} for other cases of 
non-compact manifolds including  the covering manifolds). 
 This was already observed in \cite[~\!Problem 6.1]{MM07}.
The following theorem allows to construct uniform peak sections of 
the powers $L^p$ for $p$ sufficiently large.
\begin{thm}\label{t4.0} 
Let $(X,J, g^{TX})$ be a complete Hermitian manifold.
Assume that $R^L, R^{E}$, $J$, $g^{TX}$ have bounded geometry 
and \eqref{eq:0.6} holds.
Then there exist smooth coefficients
$\bb_r(x)\in \End (E)_x$ which are polynomials in $R^{TX}$, 
$R^E$ {\rm (}and $d\Theta$, $R^L${\rm )}  and their derivatives 
with order
$\leqslant 2r-2$  {\rm (}resp. $2r-1$, $2r${\rm )} and reciprocals
of linear combinations of eigenvalues of $\dot{R}^L$ at $x$, and
\begin{align}\label{abk2.5}
 \bb_0={\det}(\dot{R}^L/(2\pi)) \Id_{E},
\end{align}
such that for any $k,\ell\in \N$, there exists
$C_{k,\ell}>0$ such that for any $p\in \N^*$,
\begin{align}\label{bk2.6}
&\Big |P_p(x,x)- \sum_{r=0}^{k} \bb_r(x) p^{n-r} 
\Big |_{\cC^\ell(X)}  \leqslant C_{k,\ell} p^{n-k-1}.
\end{align}
Moreover, the expansion is uniform in the following sense{\rm:} 
for any fixed $k,\ell\in \N$, assume that the derivatives of 
$g^{TX}$, $h^L$, $h^E$ with order $\leqslant 2n+2k+\ell+6$ 
run over a set bounded in the $\cC^\ell$--\,norm taken with 
respect to 
the parameter $x\in X$ and, moreover, $g^{TX}$ runs over a set
bounded below, then the constant $C_{k,\,\ell}$ is independent 
of $g^{TX}${\rm;}
and the $\cC^\ell$-norm in \eqref{bk2.6} includes also the  
derivatives with respect to the parameters. 
\end{thm}
We will actually make use in the following only of the case 
$\ell=0$ from Theorem \ref{t4.0}.
\begin{thm}\label{t4.1} 
Let $(X,J, g^{TX})$ be a complete Hermitian manifold.
Assume that 
$R^L, R^{E}$, $J$, $g^{TX}$ have bounded geometry and 
\eqref{eq:0.6} holds.
Then there exists $p_0\in\N$ such that for all $p\geqslant p_0$,
the bundle $L^p\otimes E$ is generated by its global holomorphic 
$L^2$-sections.
\end{thm}
\begin{proof}
By Theorem \ref{t4.0},  
\begin{equation}\label{e:4.2}
P_p(x,x)=\bb_0(x)p^n+O(p^{n-1})\,,\:\:\text{uniformly on $X$}\,,
\end{equation}
where $\bb_0=\det(\dot{R}^L/2\pi)\Id_E$. Due to \eqref{eq:0.6},
the function $\det(\dot{R}^L/2\pi)$ is bounded below 
by a positive constant. 
Hence \eqref{e:4.2} implies that there exists $p_0\in \N$
such that for all $p\geqslant p_0$ and all $x\in X$ 
the endomorphism $P_p(x,x)\in\End(E_x)$ is invertible.
In particular, for any 
$v\in L^p_x\otimes E_x$ there exists $S=S(x,v)
\in H^0_{(2)}(X,L^p\otimes E)$ with $S(x)=v$.
\end{proof}
We can apply Theorem \ref{t4.1} to the following situation.
Let $h_{p}$ be a Hermitian metric on $L^p$.
Let $\pi:\wi{X}\to X$ be a Galois covering and consider 
\begin{align}\label{eq:5.4}
\wi{\Theta}=\pi^*\Theta, \:\ dv_{\wi{X}}=\wi{\Theta}^n/n!, 
\: 
(\wi{L}^p, \wi{h}_{p})=(\pi^*L^p,\pi^*h_{p}), 
\: (\wi{E},h^{\wi{E}})=(\pi^*E,\pi^*h^E),
\end{align}
and 
let $L^2(\wi{X},\wi{L}^p\otimes\wi{E})$ be the $L^2$-space 
of sections of $\wi{L}^p\otimes\wi{E}$ with respect to 
$\wi{h}_{p}$, $h^{\wi{E}}$, $dv_{\wi{X}}$.
The space of global holomorphic $L^2$-sections is defined by 
\begin{align}\label{eq:5.5}
H^0_{(2)}(\wi{X},\wi{L}^p\otimes\wi{E})
=\{S\in L^2(\wi{X},\wi{L}^p\otimes\wi{E}):\db S=0\}\,.
\end{align}
\begin{cor}\label{t4.1a} 
Let $(X,\Theta)$ be a compact Hermitian manifold, 
$L$ be a positive line bundle over $X$. 
Then there exists  $p_0\in \N$ such that for all
Galois covering $\pi:\wi{X}\to X$, for all  $p\geqslant p_0$, 
and all Hermitian metrics $h_{p}$, $h^E$ on $L^p$, $E$, 
the bundle $\wi{L}^p\otimes\wi{E}$ is generated by
its global holomorphic $L^2$-sections.
\end{cor}
\begin{proof}
Indeed, $(\wi{X}, g^{T\wi{X}})$ is complete  and 
$R^{\wi{L}}, R^{\wi{E}}$, $\wi{J}$, $g^{T\wi{X}}$
have bounded geometry. Thus the conclusion follows immediately
from Theorem \ref{t4.1} for metrics $h_{p}$ of the form 
$(h^{L})^p$, where $h^L$ is a positively curved metric on $L$. 
That $p_0$ is independent of the covering $\pi:\wi{X}\to X$ 
follows from the dependency
conditions of $p_0$ in Theorem \ref{t4.1}.
Observe finally that the $L^2$ condition is independent of 
the Hermitian metric $h_{p}$, $h^E$ chosen on $L^p$, $E$ 
over the compact manifold $X$.
\end{proof}

Note that instead of 
using Theorem \ref{t4.1} we could have also concluded by 
using \cite[Theorem 6.1.4]{MM07}.
The latter shows that, roughly speaking, the asymptotics of 
the Bergman kernel on the base manifold and on the covering 
are the same. Note also that by Theorem \ref{t4.1} 
or \cite[Theorem 6.1.4]{MM07} we obtain an estimate from below 
of the von Neumann dimension of the space of $L^2$ holomorphic 
sections (cf.\ \cite[\!Remark 6.1.5]{MM07}):
\begin{equation}
\dim_{\Gamma}H^0_{(2)}(\wi{X},\wi{L}^p\otimes\wi{E})
\geqslant\frac{p^n}{n!}\int_X\left(\frac{\sqrt{-1}}{2\pi}R^L\right)^n
+o(p^n)\,,\:\:p\to\infty\,.
\end{equation}
This was used in \cite{TCM} and \cite[\S 6.4]{MM07} 
to obtain weak Lefschetz theorems, extending results 
from \cite{NR:98}.

The following definition was introduced in 
\cite[Definition 4.1]{Na90} for line bundles.
\begin{defn}\label{t4.2a} 
Suppose $X$ is a complex manifold, $(F,h^F)\to X$ is a 
Hermitian holomorphic vector
bundle. The manifold $X$ is called \emph{holomorphically convex
with respect to $(F,h^F)$} if, for every infinite subset $S$
without limit points in $X$, there is a holomorphic section $S$ of 
$F$ on $X$ such that $|S|_{h^F}$ is
unbounded on $S$. The manifold $X$ is called 
holomorphically convex if it is holomorphically convex 
with respect to the trivial line bundle.
\end{defn}
Since it
suffices to consider any infinite subset of $S$, in order to prove 
the holomorphic convexity we may assume that $S$ is actually 
equal to a sequence of points $\{x_i\}_{i\in\N}$ 
without limit points in $X$.
\begin{thm}\label{t4.3} 
Let $(X,J, g^{TX})$ be a complete Hermitian manifold.
Assume that 
$R^L, R^{E}$, $J$, $g^{TX}$ have bounded geometry and 
\eqref{eq:0.6} holds.
Then there exists $p_0\in\N$ such that for all
$p\geqslant p_0$, $X$ is holomorphically convex with 
respect to the bundle $L^p\otimes E$.
\end{thm}
\begin{proof} If $X$ is compact the assertion is trivial. 
    We assume in the sequel that $X$ is non-compact.
We use the following lemma.
\begin{lemma}\label{l4.5}
There exists $p_0\in\N$ such that for all $p\geqslant p_0$, 
for any compact set $K\subset X$ and any $\var,M>0$ there 
exists a compact set $K(\var,M)\subset X$ with the property 
that for any $x\in X\setminus K(\var,M)$ there exists 
$S\in H^0_{(2)}(X,L^p\otimes E)$ with
$|S(x)|\geqslant M$ and $|S|\leqslant\var$ on $K$.
\end{lemma}
\begin{proof}
Let $p_0\in\N$ be as in the conclusion of Theorem \ref{t4.1}.
For any $x\in X$, $w\in L^p_x\otimes E_x$,  consider 
the peak section 
\begin{align}\label{eq:5.7}
S\in H^0_{(2)}(X,L^p\otimes E), \quad 
y\longmapsto P_p(y,x)\cdot w. 
\end{align}
Since $P_p(x,x)$ is invertible, we can find for any given 
$v\in L^p_x\otimes E_x$ a peak section $S_{x,v}$ 
as in \eqref{eq:5.7} such 
that $S_{x,v}(x)=v$. Thus for any $x\in X$ there exists 
$v(x)\in L^p_x\otimes E_x$ such that $|S_{x,v(x)}(x)|\geqslant M$.
By \eqref{eq:0.7}, for any fixed $0<r<\inj^X$, $S_{x,v(x)}$ 
has exponential decay outside the ball $B(x,r)$, uniformly in 
$x\in X$.
We can now choose $\delta>0$, such that for any $x\in X$ with 
$d(x,K)>\delta$ we have $|S_{x,v(x)}(y)|\leqslant\var$ for 
$y\in K$. We set finally 
\begin{align}\label{eq:5.9}
K(\var,M)=\{z\in X: d(z,K)\leqslant\delta\}.
\end{align}
The proof of Lemma \ref{l4.5} is completed.
\end{proof}
In order to finish the proof of Theorem \ref{t4.3}, let us 
choose an exhaustion $\{K_i\}_{i\in\N}$ of $X$ with compact sets, 
i.\,e., $K_i\subset\mathring{K}_{i+1}$ and 
$X=\bigcup_{i\in\N} K_i$. Consider a sequence $\{x_i\}_{i\in\N}$ 
in $X$ without limit points. 
Using Lemma \ref{l4.5} we construct inductively a sequence of 
holomorphic sections $\{S_i\}_{i\in\N}$ and a subsequence 
$\{\nu(i)\}_{i\in\N}$ of $\N$ such that 
\begin{align}\label{eq:5.10}
|S_i|\leqslant 2^{-i}
    \text{ on }  K_i  \text{ and }
|S_i(x_{\nu(i)})|\geqslant 2^i+\sum_{j<i}|S_j(x_{\nu(i)})|, 
\end{align}
where $\nu(i)$ is the smallest index $j$ such that 
$x_j\in X\setminus K_i(2^{-i},2^i+\sum_{j<i}|S_j(x_{\nu(i)})|)$.
Then $S=\sum_{i\in\N} S_i$ converges uniformly on any compact 
set of $X$, hence defines a holomorphic section of $L^p\otimes E$
on $X$, and satisfies $|S(x_{\nu(i)})|\to\infty$ as $i\to\infty$.
\end{proof}
\begin{rem}\label{r4.1}
\textbf{(a)}
Napier \cite[\!Theorem 4.2]{Na90} proves a similar result for 
a complete K\"ahler manifold with bounded geometry and 
for $(E,h^E)$ trivial. His notion of bounded geometry is weaker 
than that used in the present paper 
(cf.\  \cite[\!Definition 3.1]{Na90}), so he first concludes 
the holomorphic convexity for the adjoint bundle 
$L^p\otimes K_X$ (twisting with $K_X$ is necessary 
for the application of the $L^2$ method for solving the 
$\db$-equation, due to Andreotti-Vesentini and H\"ormander, 
see e.\,g.,\ \cite[Th\'eor\`eme 5.1]{D82}, 
\cite[Theorem B.4.6]{MM07}). 
If the Ricci curvature of $g^{TX}$ is bounded from below, 
then \cite[\!Theorem 4.2 (ii)]{Na90} shows that $X$ is 
holomorphically convex with respect to $L^p$, for $p$ 
sufficiently large. Note that our notion of bounded geometry 
implies that the Ricci curvature of $g^{TX}$ is bounded from below.

\textbf{(b)} In the conditions of Theorem \ref{t4.3} there exists 
$p_0\in \N$ such that for all $p\geqslant p_0$  we have:
\begin{itemize}
\item[(i)] $H^0_{(2)}(X,L^p\otimes E)$ separate points of $X$,
i.\,e., for any $x,y\in X$, $x\neq y$, there exists
$S\in H^0_{(2)}(X,L^p\otimes E)$ with $S(x)=0$, $S(y)\neq0$.

\item[(ii)] $H^0_{(2)}(X,L^p\otimes E)$ gives local coordinates 
on $X$, i.\,e., for any $x\in X$ there exist 
$S_0$, $S_1$, \ldots, $S_n\in H^0_{(2)}(X,L^p\otimes E)$ 
such that $S_0(x)\neq0$ and $S_1/S_0$, \ldots, $S_n/S_0$ 
form a set of holomorphic coordinates around $x$.
\end{itemize}
The items (i) and (ii) follow from the $L^2$ estimates for $\db$ 
for singular metrics by using similar arguments as in 
\cite{D82}, \cite{Na90}. We show next how they follow also from 
the asymptotics of the Bergman kernel.
\end{rem}
\begin{prop}\label{p4.1}
Let $(X,J,g^{TX})$ be a complete K\"ahler manifold. 
Assume also that there exist $\varepsilon,C>0$ such that on $X$,
\begin{equation}\label{eq:0.6a}
\sqrt{-1}R^L(\cdot, J\cdot)\geqslant\varepsilon 
g^{TX}(\cdot,\cdot)\,,\:\:\:
\sqrt{-1}(R^{\det}+R^E)> -C\om \Id_E\,.
\end{equation}
Then for any compact set $K\subset X$ there exists
$p_0=p_0(K)\in \N$ such that for all
$p\geqslant p_0$, the sections of $H^0_{(2)}(X,L^p\otimes E)$ 
separate points and give local coordinates on $K$. 
\end{prop}
\begin{proof}
For $x\in X$, we construct as in \cite[\S 6.2]{MM07} 
the generalized Poincar\'e metric.
Consider the blow-up $\alpha:\widehat{X}\to X$ with center $x$ 
and denote by $D\subset\widehat{X}$ the exceptional divisor. 
By \cite[Proposition 2.1.11 (a)]{MM07},
there exists a smooth Hermitian metric $h_0$ on the line bundle 
$\cO(-D)$ whose curvature $R^{(\cO(-D),\,h_0)}$ 
is strictly positive along $D$, and vanishes outside a compact 
neighborhood of $D$.
We consider on $\widehat{X}$ the complete K\"ahler metric 
\begin{align}\label{eq:5.9a}
\Theta'=\alpha^*\om+\eta\sqrt{-1}
R^{(\cO(-D),\,h_0)}\,,\:\:0<\eta\ll 1\,.
\end{align}
The generalized Poincar\'e metric
on $X\setminus\{x\}=\widehat{X}\setminus D$ 
is defined by the Hermitian form
\begin{equation}\label{poin} 
\Theta_{\var_0}=\Theta ' +\varepsilon_0\sqrt{-1}\,
\db \partial\log\left((-\log\|\sigma\|^2)^2\right)\,, 
\quad \text{$0<\varepsilon_0\ll 1$ fixed}, 
\end{equation}   
where $\sigma$ is a holomorphic section 
of the associated  holomorphic line bundle $\cO(D)$ 
which vanish to first order on $D$, and 
$\|\sigma\|$ is the norm for a smooth Hermitian metric $\|\cdot\|$ 
on $\cO(D)$ such that $\|\sigma\|<1$.
By \cite[Lemma 6.2.1]{MM07} the generalized Poincar\'e metric is 
a complete K\"ahler metric of
finite volume in the neighborhood of $x$, whose Ricci curvature is bounded from below.
Set
\begin{align}\label{eq:5.10a}
h^L_\var:=h^L\,(-\log(\|\sigma\|^2))^\var\,e^{-\var\psi}, 
\quad 0<\varepsilon\ll 1\,,
\end{align}
where $\psi$ is a singular weight which vanishes outside 
a neighbourhood of $x$ and $\psi(z)=2n\log|z-x|$ in 
local coordinates near $x$.
We can apply \cite[Theorem 6.1.1]{MM07} to
$(X\setminus~\{x\},\Theta_{\var_0})$, $(L,h^L_{\varepsilon})$ 
and $E$, since condition \eqref{eq:0.6a} are satisfied for this data.
Let $y\in X$, $y\neq x$. Due to the asymptotics of 
the Bergman kernel on $X\setminus\{x\}$, there exists 
$p_0(x,y)$ such that for all $p\geqslant p_0(x,y)$ there is a section
\begin{align}\label{eq:5.11a}
\widehat{S}^p_{x,y}\in H^0_{(2)}(X\setminus\{x\},
L^p\otimes E,h^L_{\varepsilon},\Theta_{\var_0}^n)
\text{  with }  \widehat{S}^p_{x,y}(y)\neq 0.
\end{align}
Since the volume form 
$\Theta_{\var_0}^n$ dominates the Euclidean volume form and 
$e^{-\var\psi}$ is not integrable near $x$, the $L^2$ condition 
on $\widehat{S}^p_{x,y}$ implies that $\widehat{S}^p_{x,y}$ 
extends to $X$ with the value $0$ at $x$.
The extension ${S}^p_{x,y}$ is necessarily holomorphic by 
the Hartogs theorem and moreover an element of 
$H^0_{(2)}(X,L^p\otimes E)$. It satisfies 
${S}^p_{x,y}(x)=0$, ${S}^p_{x,y}(y)\neq0$, as desired.

In a similar manner one proves that $H^0_{(2)}(X,L^p\otimes E)$ 
separate points and give local coordinates on $K$.
\end{proof}
\begin{rem}\label{t4.11}
\textbf{(a)} Under the hypotheses of Theorem \ref{t4.3} 
(that is, bounded geometry) the argument above shows that 
the $p_0\in \N$ in Proposition \ref{p4.1} can be chosen 
such that for all $p\geqslant p_0$, $H^0_{(2)}(X,L^p\otimes E)$ 
separate points and give local coordinates on the whole $X$, 
i.\,e., points (i) and (ii) from Remark \ref{r4.1}\,(b) hold. 

\textbf{(b)} The separation of points in Proposition \ref{p4.1} 
follows also from a non-compact version of 
\cite[\!Theorem 1.8]{HsM11}, where the asymptotics of 
the Bergman kernel for space of sections of a positive line bundle
twisted with the Nadel multiplier sheaf of a singular metric
are obtained.
\end{rem}
\begin{thm}\label{t4.8}
Let $(X,J, \Theta)$ be a complete Hermitian manifold and let 
$\varphi$ be a smooth function on $X$.
Assume that $\varphi$ is bounded from below,
$\partial\db\varphi$, $J$, $g^{TX}$ have bounded geometry 
and there exists $\var>0$ such that 
$\sqrt{-1}\partial\db\varphi\geqslant\var\Theta$ on $X$.
Then $X$ is a Stein manifold.
\end{thm}
\begin{proof}
We apply Theorem \ref{t4.3} for the trivial line bundle 
$L=X\times\C$ endowed with the metric $h^L$ defined by
$|1|^2_{h^L}=e^{-2\varphi}$. Then 
\begin{align}\label{eq:5.11}
R^L=2\partial\db\varphi\,.
\end{align}
Thus $X$ is holomorphically convex with respect to 
$(L,e^{-2p\varphi})$ for $p$ sufficiently large. 
Since $\varphi$ is bounded below this implies that $X$
is holomorphically convex with respect to the trivial line bundle 
endowed with the trivial metric. Moreover, 
Remark \ref{r4.1}\,(b), or Proposition \ref{p4.1}, 
shows that global holomorphic functions on $X$ separate points 
and give local coordinates on $X$. Hence $X$ is Stein.
\end{proof}
\begin{cor}\label{t4.3a} 
Let $(X,J,\Theta)$ be a compact Hermitian manifold and
$(L,h^L)$ be a positive line bundle  over $X$.
Then there exists $p_0\in \N$ such that for all
Galois covering $\pi:\wi{X}\to X$, for all $p\geqslant p_0$, 
and all Hermitian metrics $h_{p}$\,, $h^E$ on $L^p$, $E$, $\wi{X}$
is holomorphically convex with respect to the bundle 
$\wi{L}^p\otimes\wi{E}$.
\end{cor}
\begin{proof}
The conclusion follows immediately from Theorem \ref{t4.3} for 
metrics $h_{p}$ of the form $(h^{L})^p$, where 
$h^L$ is a positively curved metric on $L$. 
Observe finally that the convexity
is independent of the Hermitian metrics $h_{p}$, $h^E$ 
chosen on $L^p$, $E$ over the compact manifold $X$.
\end{proof}
\begin{rem}\label{t4.14}
We can prove Corollary \ref{t4.3a} also without the use of
the off-diagonal decay of the Bergman kernel. What is actually 
needed is only Corollary \ref{t4.1a} which uses only the diagonal 
expansion of the Bergman kernel. In order to carry out
the proof we show that Lemma \ref{l4.5} follows from 
Corollary \ref{t4.1a}. By this Corollary, for any $x\in\wi{X}$ 
there exists
$S_x\in H^0_{(2)}(\wi{X},\wi{L}^p\otimes\wi{E})$ such that 
$|S_x(x)|\geqslant M$. 
Since $S_x\in L^2(\wi{X},\wi{L}^p\otimes\wi{E})$, there exists 
a compact set $A(S_x,\var)\subset\wi{X}$ such that
\begin{align}\label{eq:5.14}
|S_x|\leqslant\var  \text{  on }  \wi{X}\setminus A(S_x,\var).
\end{align}
Let $F\subset\wi{X}$ be a compact fundamental set. 
Using what has been said, there exists a compact set 
$F(\var,M)\supset F$ and sections 
$S_1,\ldots,S_m\in H^0_{(2)}(\wi{X},\wi{L}^p\otimes\wi{E})$ 
such that 
\begin{align}\label{eq:5.15}
\max_{1\leqslant i\leqslant m}|S_i(x)|\geqslant M
\text{  for all }  x\in F  \text{ and }
\max_{1\leqslant i\leqslant m}|S_i(x)|\leqslant\var
\text{  for all }   x\in\wi{X}\setminus F(\var,M). 
\end{align}
Let $K\subset\wi{X}$ 
be a compact set and let $\Gamma$ be the group of 
deck transformations of $\pi:\wi{X}\to X$. Define 
\begin{align}\label{eq:5.16}
K(\var,M)=\bigcup\{\gamma F:\gamma\in\Gamma,\,
K\cap\gamma F(\var,M)\neq\emptyset\}\,.
\end{align}
Consider now $x\in\wi{X}\setminus K(\var,M)$. 
Then there is $\gamma\in\Gamma$ such that $\gamma^{-1}x\in F$.
It follows that $K\cap\gamma F(\var,M)=\emptyset$ so there is 
$S_i$ such that $|\gamma S_i(x)|\geqslant M$ and 
$|\gamma S_i|\leqslant\var$ on $K$. 
\end{rem}


\def\cprime{$'$} \def\cprime{$'$} \def\cprime{$'$}
\providecommand{\bysame}{\leavevmode\hbox to3em{\hrulefill}\thinspace}
\providecommand{\MR}{\relax\ifhmode\unskip\space\fi MR }
\providecommand{\MRhref}[2]{%
  \href{http://www.ams.org/mathscinet-getitem?mr=#1}{#2}
}
\providecommand{\href}[2]{#2}

    
\end{document}